\documentclass[preprint,12pt]{elsarticle}

\usepackage{amssymb}

\usepackage{amstext,amssymb,amsmath,amsbsy, dsfont}

\usepackage{hyperref}
\usepackage{cleverref}

\usepackage{amscd,dsfont}
\usepackage{amsfonts}
\usepackage{indentfirst}
\usepackage{verbatim}
\usepackage{amsmath}
\usepackage{amsthm}
\usepackage{enumerate}
\usepackage{graphicx}
\usepackage{color}
\usepackage[OT1]{fontenc}
\usepackage[latin1]{inputenc}
\usepackage[english]{babel}
\usepackage{amssymb,esint}
\newtheorem{theorem}{Theorem}
\newtheorem{lemma}{Lemma}

\newtheorem{remark}{Remark}

\newcommand{\C}{\kappa}

\newcommand{\mN}{\mathbb{N}}

\newcommand{\mZ}{\mathbb{Z}}

\newcommand{\mR}{\mathbb{R}}
\newcommand{\mS}{\mathbb{S}}
\newcommand{\eps}{\varepsilon}

\newcommand{\dsp}{\displaystyle}

\newcommand{\tvarphi}{\widetilde \varphi}

\numberwithin{equation}{section}

\begin{document}

\begin{frontmatter}

\title{Non-local, non-convex functionals converging to Sobolev norms}
\author[add1,add2,add3]{HA\"IM  BREZIS} 

\ead{brezis@math.rutgers.edu}
\address[add1]{Department of Mathematics, 
Rutgers University, Hill Center, Busch Campus, 
	110 Frelinghuysen Road, Piscataway, NJ 08854, USA}
\address[add2]{Departments of Mathematics and Computer Science, Technion, Israel Institute of Technology, 32.000 Haifa, Israel}
\address[add3]{Laboratoire Jacques-Louis Lions,
Sorbonne Universit\'es, UPMC Universit\'e Paris-6, 4  place Jussieu, 
75005 Paris, France,}

\author[add4]{HOAI-MINH NGUYEN}
\ead{hoai-minh.nguyen@epfl.ch}

\address[add4]{Department of Mathematics, EPFL SB CAMA, 	Station 8 CH-1015 Lausanne, Switzerland}

\begin{abstract} We study the pointwise convergence and the $\Gamma$-convergence of a family of non-local, non-convex functionals $\Lambda_\delta$ in $L^p(\Omega)$ for $p>1$. We show that the limits are  multiples of  $\int_{\Omega} |\nabla u|^p$. This is a continuation of our previous work where the case $p=1$ was considered. 
\end{abstract}


\end{frontmatter}





\section{Introduction and statement of the main results}

Assume that  $\varphi:[0, +\infty) \to [0, + \infty)$ is defined at {\it every} point of $[0, + \infty)$, $\varphi$ is
continuous on $[0, +\infty)$ except at a finite number of points in $(0, +\infty)$ where it admits a limit from the left and from the right, and  $\varphi(0) = 0$. Let $\Omega \subset \mR^d$ ($d \ge 1$) denote a domain which  is either bounded and smooth, or $\Omega=\mR^d$.  Given a measurable function $u$ on $\Omega$, and a  parameter $\delta > 0$, we define the following non-local functionals, for $p>1$, 
\begin{equation}\label{def-Lambda}
\Lambda (u, \Omega): = \int_\Omega \int_\Omega \frac{\varphi(|u(x) - u(y)|) }{|x - y|^{p + d}} \, dx \,
dy \quad \mbox{and} \quad \Lambda_{\delta}(u, \Omega): = \delta^p \Lambda (u/\delta, \Omega).
\end{equation}
To simplify the notation, we will often delete $\Omega$ and write 
 $\Lambda_\delta (u)$ instead of $\Lambda_\delta(u, \Omega)$. 

As in \cite{BrNg-Nonlocal1},  we consider the following four assumptions on $\varphi$:
\begin{equation}\label{cond-varphi-0}
\varphi(t) \le a t^{p+1} \mbox{ in } [0,1] \mbox{ for some positive constant } a,
\end{equation}
\begin{equation}\label{cond-varphi-1}
\varphi(t) \le b  \mbox{ in } \mR_{+} \mbox{ for some positive constant } b,
\end{equation}
\begin{equation}\label{cond-varphi-2}
\varphi \mbox{ is non-decreasing},
\end{equation}
and
\begin{equation}\label{cond-varphi-3}
\gamma_{d, p}\int_0^\infty \varphi(t) t^{-(p+1)} \,dt=1, \, \mbox{where } \gamma_{d, p}: = \int_{\mS^{d-1}} |\sigma \cdot e |^p \, d \sigma \mbox{ for some } e \in \mS^{d-1}.
\end{equation}


In this paper, we study  the pointwise and the $\Gamma$-convergence of $\Lambda_\delta$ as $\delta \to 0$ for $p>1$. This is a continuation of our previous work \cite{BrNg-Nonlocal1} where the case $p =1$ was investigated in great details.  Concerning the pointwise convergence of $\Lambda_{\delta}$, our main result is 

\begin{theorem}\label{thm-pointwise} Let $d   \ge 1$ and  $p>1$. 
Assume \eqref{cond-varphi-0}, \eqref{cond-varphi-1}, and \eqref{cond-varphi-3} (the monotonicity assumption \eqref{cond-varphi-2} is not required here).  We have

\begin{enumerate}
\item[i)] There exists a positive constant $C_{p, \Omega}$ such that 
\begin{equation}\label{thm-p-1}
\Lambda_\delta(u, \Omega) \le C_{p, \Omega} \int_{\Omega} |\nabla u|^p \, dx \quad \forall \, 
u \in W^{1, p}(\Omega),   \forall \,  \delta > 0; 
\end{equation}
moreover, 
\begin{equation}\label{thm-p-2}
\lim_{\delta \to 0} \Lambda_\delta(u, \Omega) = \int_{\Omega} |\nabla u|^p \, dx \quad \forall \, 
u \in W^{1, p}(\Omega). 
\end{equation}

\item[ii)] Assume in addition that $\varphi$ satisfies  \eqref{cond-varphi-2}. Let  $u \in L^p(\Omega)$ be such that 
\begin{equation}\label{thm-p-3}
\liminf_{\delta \to 0} \Lambda_\delta (u, \Omega) < + \infty, 
\end{equation}
then $u \in W^{1, p}(\Omega)$. 
\end{enumerate}
\end{theorem}

\begin{remark} \rm \Cref{thm-pointwise} provides a characterization of the Sobolev space $W^{1, p}(\Omega)$ for $p>1$: 
$$
W^{1, p} (\Omega) = \Big\{u \in L^p(\Omega); \liminf_{\delta \to 0} \Lambda_\delta(u) < + \infty \Big\}. 
$$ 
This fact is originally due to Bourgain and Nguyen \cite{BourNg, NgSob1} when $\varphi  = \hat \varphi_1:  = c \mathds{1}_{(1, + \infty)}$ for 
 an appropriate   constant $c$. 

There are some similarities but also striking differences between the cases $p>1$ and $p=1$. 

\noindent a) First note a similarity. Let $p = 1$ and $\varphi$ satisfy \eqref{cond-varphi-0}-\eqref{cond-varphi-2}, and assume that $u \in L^1(\Omega)$ verifies 
 $$
\liminf_{\delta \to 0} \Lambda_{\delta}(u, \Omega) < + \infty, 
$$
then $ u \in BV(\Omega)$ (see \cite{BourNg, BrNg-Nonlocal1}). 

\noindent b) Next is a major difference.  Let $p=1$. There exists $u \in W^{1, 1}(\Omega)$ such that, for all  $\varphi$ satisfying \eqref{cond-varphi-0}-\eqref{cond-varphi-2}, one has 
$$
\lim_{\delta \to 0} \Lambda_{\delta}(u, \Omega) = + \infty
$$
\cite[Pathology 1]{BrNg-Nonlocal1}. In particular, \eqref{thm-p-1} and \eqref{thm-p-2} do not hold for $p=1$. 
An example in the same spirit was originally constructed by Ponce and is presented in \cite{NgSob1}. Other pathologies occurring in the case $p=1$ can be found in \cite[Section 2.2]{BrNg-Nonlocal1}. 

As we will see later, the proof of \eqref{thm-p-1} involves the theory of maximal functions. The use of this theory was suggested independently by Nguyen \cite{NgSob1} and Ponce and van Schaftingen (unpublished communication to the authors).  The proof of \eqref{thm-p-1}  uses the same strategy as  in \cite{NgSob1}. 

We point out that assertion $ii)$ fails without the monotonicity condition \eqref{cond-varphi-2} on $\varphi$.  Here is an example e.g. with  $\Omega = \mR$. Let $\varphi = c \mathds{1}_{(1, 2)}$ for an appropriate,  positive constant $c$. Let $u = \mathds{1}_{(0, 1)}$. One can easily check that $\Lambda_{\delta} (u) = 0$ for $\delta \in (0, 1/2)$ and  it is clear that $u \not \in W^{1, p}(\mR)$ for $p>1$.   

\end{remark}

Concerning the $\Gamma$-convergence of $\Lambda_{\delta}$, our main result is

\begin{theorem}\label{thm-gamma} Let $d \ge 1$ and $p > 1$. Assume \eqref{cond-varphi-0}-\eqref{cond-varphi-3}. Then
\begin{equation*}
\Lambda_{\delta}(\cdot,  \Omega) \; \Gamma\mbox{-converges in $L^p(\Omega)$ to }  \Lambda_0(\cdot, \Omega) : = \C \int_{\Omega} |\nabla \cdot |^p \, dx, 
\end{equation*}
as $\delta \to 0$,  for some constant $\C$ which depends only on $p$ and $\varphi$, and verifies
\begin{equation}\label{constant-gamma}
0 <   \C \le 1.
\end{equation}
\end{theorem}

 \Cref{thm-gamma} was known earlier when  $\varphi = 
\hat \varphi_1 $ \cite{NgGammaCRAS, NgGamma}.

\medskip 
The paper is organized as follows. \Cref{thm-pointwise} is proved in \Cref{sect-pointwise} and the proof of \Cref{thm-gamma} is given in \Cref{sect-gamma}. Throughout the paper, we denote
$$
\varphi_\delta(t) := \delta^p \varphi(t/\delta) \mbox{ for } p>1, \delta > 0, t \ge 0. 
$$

\section{Proof of Theorem~\ref{thm-pointwise}}\label{sect-pointwise}

In view of the fact that  $\liminf_{t \to + \infty} \varphi (t) > 0$, assertion \eqref{thm-p-3} is a direct consequence of \cite[Theorem 1]{BourNg}; note that   \cite[Theorem 1]{BourNg} is stated for $\Omega = \mR^d$ but the proof can be easily adapted to the case where $\Omega$ is bounded. It could also be deduced from \Cref{thm-gamma}. 

We now establish assertions \eqref{thm-p-1} and \eqref{thm-p-2}. The proof consists of two steps. 

{\it Step 1}: Proof of  \eqref{thm-p-1} and \eqref{thm-p-2} when $\Omega  = \mR^d$ and $u \in W^{1, p}(\mR^d)$.  Replacing $y$ by $x + z$ and using polar coordinates in the $z$ variable, we find
\begin{equation}\label{tt1}
\int_{\mR^d} \, dx  \int_{\mR^d} \frac{\varphi_\delta(|u(x) - u(y)|)}{|x -y|^{p + d}} \, dy = \int_{\mR^d} \, dx 
\int_0^{+\infty}   \, dh  \int_{\mS^{d-1}} \frac{\varphi_\delta(|u(x + h \sigma) - u(x)|)}{h^{p+1}} \, d \sigma .
\end{equation}
We have
\begin{multline}\label{tt2}
\int_{\mR^d} \, dx \int_0^{+\infty}\, d h \int_{\mS^{d-1}} \frac{\varphi_\delta(|u(x + h \sigma) - u(x)|)}{h^{p+1}} \, d
\sigma \\[6pt]
= \int_{\mR^d} \, dx  \int_0^{+\infty} \, d h \int_{\mS^{d-1}} \frac{\delta^p \varphi \Big(|u(x + h \sigma) - u(x)| \big/
\delta \Big)}{h^{p+1}} \, d \sigma.
\end{multline}
Rescaling the variable $h$ gives
\begin{multline}\label{tt3}
\int_{\mR^d} \, dx \int_0^{+\infty} \, dh  \int_{\mS^{d-1}}  \frac{\delta^p \varphi \Big(|u(x + h \sigma) - u(x)| \big/
\delta \Big)}{h^{p+1}}  \, d \sigma\\[6pt]
= \int_{\mR^d} \, dx  \int_0^{+\infty} \, d h \int_{\mS^{d-1}}\frac{ \varphi \Big(|u(x + \delta h \sigma) - u(x)| \big/
\delta \Big)}{h^{p+1}} \, d \sigma  .
\end{multline}
Combining \eqref{tt1}, \eqref{tt2}, and \eqref{tt3} yields
\begin{equation}\label{tt4}
\int_{\mR^d} \ dx  \int_{\mR^d} \frac{\varphi_\delta(|u(x) - u(y)|)}{|x -y|^{d+p}} \, dy = \int_{\mR^d} \, dx
\int_0^{+\infty} \, d h \int_{\mS^{d-1}} \frac{ \varphi \Big(|u(x + \delta h \sigma) - u(x)| \big/
\delta \Big)}{h^{p+1}}  \, d \sigma .
\end{equation}
Note that
\begin{equation}\label{tt5}
\lim_{\delta \to 0}\frac{|u(x + \delta h \sigma) - u(x)|}{\delta}  = |\langle \nabla u(x),  \sigma \rangle | h  \mbox{ for a.e. 
} (x, \, h, \, \sigma) \in \mR^d \times [0, + \infty) \times \mS^{d-1}.
\end{equation}
Here and in what follows, $\langle ., . \rangle$ denotes the usual scalar product in $\mR^d$.  
Since $\varphi$ is continuous at $0$ and on $(0, +\infty)$ except at a finite number of points, it follows that
\begin{multline}\label{tt6}
\lim_{\delta \to 0}\frac{1}{h^{p+1}} \varphi \Big(|u(x + \delta h \sigma) - u(x)|\big/\delta \Big) =
\frac{1}{h^{p+1}}\varphi \Big(|\langle \nabla u(x),  \sigma \rangle| h \Big) \\[6pt]
\mbox{ for  a.e. } (x, \, h, \, \sigma) \in \mR^d \times (0, + \infty) \times \mS^{d-1}. 
\end{multline}
Rescaling once more the variable $h$ gives
\begin{equation}\label{tt7}
 \int_{0}^{\infty}  \, d h \int_{\mS^{d-1}} \frac{1}{h^{p+1}} \varphi\Big( | \langle \nabla u (x),  \sigma \rangle | h\Big)  \, d \sigma = |\nabla u(x)|^p \int_{0}^{\infty} \varphi(t) t^{-(p+1)}  \, dt \; \int_{\mS^{d-1}} |\langle \sigma,  e \rangle|^p \, d \sigma ;
\end{equation}
here we have also used the obvious fact that, for every $V \in \mR^d$, and for any fixed $e \in \mS^{d-1}$,
\begin{equation*}
\int_{\mS^{d-1}} | \langle V,  \sigma \rangle |^p \, d \sigma = |V|^p \int_{\mS^{d-1}} |\langle e, \sigma \rangle|^p \, d \sigma. 
\end{equation*}
Thus, by the normalization condition \eqref{cond-varphi-3}, we obtain
\begin{equation}\label{tt8}
\int_{\mR^d} \, dx  \int_{0}^\infty  \, d h \int_{\mS^{d-1}} \frac{1}{h^{p+1}} \varphi\Big( | \langle \nabla u (x), \sigma \rangle| h\Big)  \, d \sigma
 =  \int_{\mR^d} |\nabla u|^p \, dx.
\end{equation}

Set 
\begin{equation*}
\tvarphi (t) = \left\{\begin{array}{cl}
a t^{p+1} & \mbox{ for } t \in [0, 1), \\[6pt]
b & \mbox{ for } t \in [1, + \infty). 
\end{array}\right.
\end{equation*}
Then 
\begin{equation}\label{tt9-0}
\mbox{$\tvarphi$ is non-decreasing and $\varphi \le \tvarphi$}.  
\end{equation}
Note that,  for a.e. $(x, \, h, \, \sigma) \in \mR^d \times (0, + \infty) \times \mS^{d-1}$, 
\begin{equation}\label{tt9}
\frac{|u(x + \delta h \sigma) - u(x)|}{\delta} \le \frac{1}{\delta}\int_{0}^{ h \delta} | \langle \nabla u (x + s \sigma), \sigma \rangle | \, ds  \le h
M(\nabla u, \sigma)(x),
\end{equation}
where
\begin{equation*}
M (\nabla u, \sigma) (x) : = \sup_{t > 0}\frac{1}{t}\int_{0}^{t} | \langle \nabla u (x + s \sigma),  \sigma \rangle |  \, ds.
\end{equation*}
Combining \eqref{tt4} and \eqref{tt9}, we derive from \eqref{tt9-0} that 
\begin{multline}\label{tt-max-0}
\Lambda_\delta (u) \le \int_{\mS^{d-1}} \int_{\mR^d} \int_0^\infty \frac{\tvarphi( h |M(\nabla u, \sigma)(x)| )}{h^{p+1}} \, dh  \, dx \, d \sigma \\[6pt]
= \int_0^{+\infty} \tvarphi (t) t^{-(p+1)} \, dt  \int_{\mS^{d-1}} \int_{\mR^d} |M(\nabla u, \sigma)(x)|^p \, dx \, d \sigma. 
\end{multline}

We claim that, for $\sigma \in \mS^{d-1}$, 
\begin{equation}\label{tt-max}
\int_{\mR^d} |M(\nabla u, \sigma)(x)|^p \, dx   \le  C_p \int_{\mR^d} |\nabla u (x)|^p \, dx. 
\end{equation}
For notational ease, we will only consider the case  $\sigma = e_1$. By the theory of maximal functions (see e.g. \cite{SteinHarmonic}), one has, for $g \in L^p(\mR)$, 
$$
\int_{\mR} \left|\sup_{t > 0} \fint_{\xi-t}^{\xi + t} |g(s)| \, ds \right|^p \,  d \xi  \le C_p \int_{\mR} |g(\xi)|^p \, d \xi. 
$$
Using this inequality with $g(x_1) = \partial_{x_1} u(x_1, x')$ for $x' \in \mR^{d-1}$, we obtain
$$
\int_{\mR} |M(\nabla u, e_1)(x_1, x')|^p \, dx_1 \le C_p \int_{\mR} |\partial_{x_1} u(x_1, x')|^p \, dx_1. 
$$
Integrating with respect to $x'$ yields 
$$
\int_{\mR^d} |M(\nabla u, e_1)(x)|^p \, dx \le C_p \int_{\mR^{d-1}} \int_{\mR} |\partial_{x_1} u (x_1, x')|^p \, dx_1 \, dx' \le  C_p \int_{\mR^d} |\nabla u (x)|^p \, dx,
$$
and \eqref{tt-max} follows. 

Using \eqref{tt-max}, we deduce from \eqref{tt-max-0} that 
$$
\Lambda_\delta (u) \le C_{p, d} \int_{\mR^d} |\nabla u|^p \, dx, 
$$
which is \eqref{thm-p-1}. From \eqref{tt6}, \eqref{tt7}, \eqref{tt8}, and \eqref{tt9}
we derive, using the dominated convergence theorem, that 
\begin{equation*}
\lim_{\delta \to 0} \Lambda_\delta (u) =  \int_{\mR^d} |\nabla u|^p \, dx. 
\end{equation*}
This completes Step 1.

\medskip
{\it Step 2}: Proof of  \eqref{thm-p-1} and \eqref{thm-p-2} when $\Omega$ is bounded and $u \in W^{1, p}(\Omega)$.  We first claim that 
\begin{equation}\label{tt12}
\lim_{\delta \to 0} \Lambda_\delta(u) =  \int_{\Omega} |\nabla u|^p \mbox{ for } u \in W^{1, p}(\Omega). 
\end{equation}
Indeed, consider an extension of $u$ in $\mR^d$ which belongs to $W^{1, p}(\mR^d)$, and is still denoted by $u$. By the same method as in the case $\Omega = \mR^d$, we have
\begin{equation}\label{tt12-1}
\lim_{\delta \to 0}\int_{\Omega} \, dx  \int_{\mR^d} \frac{\varphi_\delta(|u(x) - u(y)|)}{|x -y|^{p + d}} \, dy = \int_{\Omega} |\nabla u|^p dx 
\end{equation}
and, for $D \Subset \Omega$ and $\eps > 0$,  
\begin{equation}\label{tt12-2}
\lim_{\delta \to 0}\int_{D} \, dx  \int_{B(x, \eps)} \frac{\varphi_\delta(|u(x) - u(y)|)}{|x -y|^{p + d}} \, dy = \int_{D} |\nabla u|^p \,  dx. 
\end{equation}
Combining \eqref{tt12-1} and \eqref{tt12-2} yields \eqref{tt12}. 

We next show that 
\begin{equation}\label{tt13}
 \Lambda_\delta(u) \le  C_{p, \Omega} \int_{\Omega} |\nabla u|^p \, dx  \mbox{ for } u \in W^{1, p} (\Omega). 
\end{equation}
Without loss of generality, we may assume that $\int_{\Omega} u = 0$. 
Consider an extension $U$ of $u$ in $\mR^d$ such that 
$$
\int_{\mR^d} |\nabla U|^p \, dx \le C_{p, \Omega} \int_{\Omega} |\nabla u|^p  \, dx. 
$$
Such an extension exists since $\Omega$ is smooth and $\int_{\Omega} u = 0$, see, e.g., \cite[Chapter 9]{BrAnalyse1}.  Using the fact
$$
\Lambda_{\delta}(u, \Omega) \le \Lambda_\delta (U, \mR^d) \le C_{p, d} \int_{\mR^d} |\nabla U|^p \, dx, 
$$
we get \eqref{tt13}.  The proof is complete. \qed

\section{Proof of Theorem~\ref{thm-gamma}}\label{sect-gamma}

We first recall  the meaning of $\Gamma$-convergence. One says that  $\Lambda_\delta(\cdot, \Omega)$ $\dsp \mathop{\to}^{\Gamma} \Lambda_0(\cdot, \Omega)$ in $L^p(\Omega)$ as $\delta \to 0$ if  
\begin{enumerate}
\item[(G1)] For each $g \in L^p(\Omega)$ and for {\it every} family
$(g_\delta) \subset L^p(\Omega)$ such that
$(g_\delta)$ converges to $g$ in $L^p(\Omega)$ as $\delta \to 0$,
one has
\begin{equation*}
\liminf_{\delta \to 0} \Lambda_\delta(g_\delta, \Omega) \ge \Lambda_0(g, \Omega).
\end{equation*}
\item[(G2)] For each $g \in L^p(\Omega)$, there {\it exists} a family
$(g_\delta) \subset L^p(\Omega)$ such that
$(g_\delta)$ converges to $g$ in $L^p(\Omega)$ as $\delta \to 0$,
and
\begin{equation*}
\limsup_{\delta \to 0} \Lambda_\delta(g_\delta, \Omega) \le \Lambda_0(g, \Omega).
\end{equation*}
\end{enumerate}

Denote $Q$ the unit open cube, i.e., $Q = (0, 1)^d$ and set 
$$
U(x) =d^{-1/2} \sum_{j=1}^d x_j \mbox{ in } Q,  
$$
so that $|\nabla U| = 1$ in $Q$.

In the following two subsections, we establish   properties (G1) and (G2) 
where $\C$ is the constant defined by 
\begin{equation}\label{def-k}
\C =  \inf \liminf_{\delta \to 0} \Lambda_{\delta} (v_\delta, Q). 
\end{equation}
Here the infimum is taken over all families of  functions $(v_\delta) \subset L^p(Q)$ such that $v_\delta \to U$ in $L^p(Q)$ as $\delta \to 0$.

\subsection{Proof of Property (G1)} 

We begin with

\begin{lemma}\label{lem-affine1}
Let $d \ge 1$, $p >  1$,  $S$ be an open bounded subset of $\mR^d$ with  Lipschitz boundary,  and let  $g$ be an affine function. Then
\begin{equation}\label{k-S}
\inf  \liminf_{\delta \to 0}\Lambda_\delta(g_\delta, S)  =  \C |\nabla g|^p |S|,
\end{equation}
where the infimum is taken over all families  $(g_\delta) \subset L^p(S)$  such
that $g_\delta \to g$ in $L^p(S)$ as $\delta \to 0$. 
\end{lemma}

\begin{proof}
The proof of Lemma~\ref{lem-affine1} is based on the definition of $\C$ in \eqref{def-k} and a covering argument. It is identical  to the one of 
 the first part of \cite[Lemma 6]{BrNg-Nonlocal1}. The details are omitted. 
\end{proof}

The proof of  Property (G1) for $p>1$ relies on the following lemma with  roots in \cite{NgGamma}.  

\begin{lemma}\label{lem-Lp}
Let $d \ge 1$, $p >  1$, and $\eps >0$. There exists two positive constants $ \hat \delta_1, \hat \delta_2$ such that for every open cube  $\widetilde Q$ which is an image of $Q$ by a dilation, for every $a \in \mR^d$, every $b\in \mR$,  and every $h \in L^p(\widetilde Q)$ satisfying 
\begin{equation}
\fint_{\widetilde Q} |h(x) - (\langle a, x\rangle + b)|^p \, dx \le \hat \delta_1 |a|^p |\widetilde Q|^{p/ d},
\end{equation}
one has 
\begin{equation}
\Lambda_{\delta} (h, \widetilde Q) \ge (\C -  \eps)|a|^p |\widetilde Q| \mbox{ for } \delta \in (0, \hat \delta_2 |a| |\widetilde Q|^{1/d}). 
\end{equation}
\end{lemma}

Hereafter,  as usual, we denote $\fint_{A} f = \frac{1}{|A|} \int_A f$. 

\begin{proof} By a change of variables, without loss of generality, it suffices to prove Lemma~\ref{lem-Lp} in the case $\widetilde Q = Q$, $|a| = 1$,  and $b=0$. We prove this by contradiction. Suppose that this
is not true. There exist $ \eps_0>0$, a sequence of measurable
functions $(h_n) \subset L^p(Q)$, a sequence $(a_n)
\subset \mR^d$,  and a
sequence $(\delta_n)$ converging to 0 such that $| a_n
|=1$,
\begin{equation*}
\int_{Q} |h_n(x) - \langle a_n, x \rangle|^p  \le \frac{1}{n},  \quad \mbox{ and } \quad \Lambda_{\delta_n}(h_n, Q) <   \C - \eps_0. 
\end{equation*}
Without loss of generality, we may assume that $(a_{n})$ converges
to $a$ for some $a \in \mR^d$ with $|a| =1$. It follows that  $( h_{n} ) $ converges to $\langle a, .
\rangle $ in $L^p(Q)$.  Applying  Lemma~\ref{lem-affine1}  with $S = Q$ and $g = \langle a , \cdot \rangle$, we obtain a contradiction.  The conclusion follows. 
\end{proof}

The second key ingredient in the proof of Property (G1) is the following useful property of functions in $W^{1, p}(\mR^d)$.

\begin{lemma} \label{lem-lem} Let $d \ge 1$, $p > 1$,  and $u \in W^{1, p}(\mR^d)$. Given $\eps_1 > 0$, there exist a subset $B = B(\eps_1)$ of Lebesgue points of $u$ and  $\nabla u$, and an integer  $\ell  = \ell (\eps_1)\ge 1$ such that 
\begin{equation}\label{G1-p-pro-Am}
\int_{\mR^d \setminus B} |\nabla u|^p \, dx \le \eps_1 \int_{\mR^d} |\nabla u|^p \, dx, 
\end{equation}
and, for every  open cube $Q'$ with $|Q'|^{1/d} \le 1/\ell$ and  $Q' \cap B \neq \emptyset$,   and for every $x \in Q' \cap B$, 
\begin{equation}\label{lem-lem-1}
\frac{1}{|Q'|^p }\fint_{Q'} \big|u(y) - u(x) - \langle \nabla u (x), y -x \rangle \big|^p \, d y \le  \eps_1  
\end{equation}
and 
\begin{equation}\label{lem-lem-2}
|\nabla u(x)|^p \ge (1 - \eps_1) \fint_{Q'} |\nabla u(y)|^p \, dy. 
\end{equation}
\end{lemma}

\begin{proof}
We first recall the following property of $W^{1, p}(\mR^d)$ functions (see e.g., \cite[Theorem 3.4.2]{Ziemer}): for a.e. $x \in \mR^d$,
\begin{equation}\label{G1-p-measure}
\lim_{r \to 0} \frac{1}{r^p} \mathop{\fint}_{Q(x, r)} \big|u(y) - u(x) - \langle \nabla u(x), y - x \rangle \big|^p \, dy = 0, 
\end{equation}
where $Q(x, r): = x + (-r, r)^d$ for $x \in \mR^d$ and $r>0$. 

Given $n \in \mN$, define, for a.e. $x \in \mR^d$,
\begin{equation}\label{G1-p-def-rho}
\rho_n(x) = \sup \left\{  \frac{1}{r^{p}}
\mathop{\fint}_{Q(x, r)} \big|u(y) - u(x) - \langle \nabla u(x), y-x \rangle \big|^p \, dy ; \; r  \in (0, 1/n) \right\}
\end{equation}
and
\begin{equation}\label{G1-p-def-tau}
\tau_n(x) = \sup  \left\{ \mathop{\fint}_{Q(x, r)}| \nabla u(y) -  \nabla u(x)|^p \, dy;  r \in (0, 1/n)  \right\}. 
\end{equation}
Note that, by  \eqref{G1-p-measure}, $\rho_n(x) \to 0$ for a.e. $x \in \mR^d$ as $n \to + \infty$. We also have, $\tau_n (x) \to 0$ for a.e. $x \in \mR^d$ as $n \to + \infty$ (and in fact at every Lebesgue point of $\nabla u$). For $m \ge 1$, set 
$$
D_m = \Big\{x \in (-m, m)^d; \mbox{$x$ is a Lebesgue point of $u$ and $\nabla u$, and }  |\nabla u(x)| \ge 1/ m \Big\}. 
$$
Since 
\begin{equation*}
\lim_{m \to + \infty}  \int_{\mR^d \setminus D_m} |\nabla u|^p \, dx = 0, 
\end{equation*}
there exists $m \ge 1$ such that  
\begin{equation}\label{G1-p-pro-Am-1}
\int_{\mR^d \setminus D_m} |\nabla u|^p \, dx \le \frac{\eps_1}{2} \int_{\mR^d} |\nabla u|^p \, dx.  
\end{equation}
Fix such an $m$.  By Egorov's theorem, there exists a  subset $B \subset D_m$ such that  $(\rho_n)$ and $(\tau_n)$ converge to $0$ uniformly on $B$, and 
\begin{equation}\label{G1-p-pro-Am-2}
\int_{D_m \setminus B} |\nabla u|^p \, dx \le \frac{\eps_1}{2} \int_{\mR^d} |\nabla u|^p \, dx.  
\end{equation}
Combining \eqref{G1-p-pro-Am-1} and \eqref{G1-p-pro-Am-2} yields \eqref{G1-p-pro-Am}. 

By the triangle inequality, we have, for every non-empty, open cube $Q'$ and a.e. $x \in \mR^d$ (in particular for $x \in Q' \cap B$),  
\begin{equation}\label{lem-lem-*}
\left(\fint_{Q'} |\nabla u(y)|^p \, dy \right)^{1/p} \le \left(\fint_{Q'} |\nabla u(y) - \nabla u(x) |^p \, dy \right)^{1/p} + |\nabla u(x)| \le \frac{|\nabla u(x)|}{(1 -\eps_1)^{1/p}}, 
\end{equation}
provided 
$$
\left(\fint_{Q'} |\nabla u(y) - \nabla u(x) |^p \, dy \right)^{1/p}  \le \left( \frac{1}{(1 -\eps_1)^{1/p}} - 1 \right) 1/m \quad \mbox{ and } \quad |\nabla u(x)| \ge 1/m.
$$
Since $(\rho_n)$ and $(\tau_n)$ converge to $0$ uniformly on $B$ and $|\nabla u(x)| \ge 1/ m$  for $x \in B$, it follows from \eqref{lem-lem-*} that there exists an $\ell \ge 1$ such that \eqref{lem-lem-1} and  \eqref{lem-lem-2} hold when $|Q'|^{1/d} \le 1/\ell$ and  $Q' \cap B \neq \emptyset$,   and $x \in Q' \cap B$. The proof is complete. 
\end{proof}

We are ready to give the

\begin{proof}[Proof of Property (G1)]  
We only consider the case $\Omega = \mR^d$. The other case can be handled as in \cite{BrNg-Nonlocal1} and is left to the reader. We follow the same strategy as in \cite{NgGamma}.

In order to  establish Property (G1), it suffices to prove that 
\begin{equation}
\liminf_{k \to + \infty} \Lambda_{\delta_k}(g_k, \mR^d) \ge \C \int_{\mR^d} |\nabla g|^p \,dx 
\end{equation}
for every $g \in L^p(\mR^d)$, $(\delta_k) \subset \mR_+$ and $(g_k) \subset L^p(\mR^d)$ such that $\delta_k \to 0$ and  $g_k \to g$ in $L^p(\mR^d)$.

Without loss of generality, we may assume that $\liminf_{k \to + \infty} \Lambda_{\delta_k}(g_k, \mR^d) < + \infty$. It follows from \cite{NgGamma} that $g \in W^{1, p}(\mR^d)$.  Fix $\eps > 0$ (arbitrary) and  let $\hat \delta_1$  be
the positive constant in \Cref{lem-Lp}. Set, for $m \ge 1$, 
$$
A_m = \Big\{x \in \mR^d; \; \mbox{ $x$ is a Lebesgue point of $g$ and  $\nabla g$, and } |\nabla g(x)| \le 1/m \Big\}.
$$
Since 
\begin{equation*}
\lim_{m \to + \infty}  \int_{A_m} |\nabla g|^p \, dx = 0, 
\end{equation*}
there exists $m \ge 1$ such that
\begin{equation}\label{G1-p-pro-Am-11}
\int_{A_m} |\nabla g|^p \, dx \le \frac{\eps}{2} \int_{\mR^d} |\nabla g|^p \, dx.  
\end{equation}
Fix such an integer $m$. 
 By Lemma~\ref{lem-lem} applied to $u = g$ and $\eps_1 = \min\{\eps/2, \delta_1/ (2m)^p \}$, there exist a subset $B$ of Lebesgue points of $g$ and $\nabla g$, and a positive integer $\ell$ such that 
\begin{equation}\label{estBm-11}
\int_{\mR^d \setminus B} | \nabla g|^p \, dx \le  \eps_1 \int_{\mR^d}
|\nabla g|^p \, dx \le  \frac{\eps}{2} \int_{\mR^d}
|\nabla g|^p \, dx, 
\end{equation}
and for every open cube  $Q'$ with $|Q'|^{1/d} \le 1 / \ell$ and  $Q' \cap B \neq \emptyset$,  and,  for every $x \in Q' \cap B$,  
\begin{equation}\label{G1-p-rho-p1}
\frac{1}{|Q'|^{p/d}}
\fint_{Q'} \big|g(y) - g(x) - \langle \nabla g(x), y - x \rangle   \big|^p \, dy \le \eps_1 \le  \hat \delta_1/ (2m)^p
\end{equation}
and
\begin{equation}\label{G1-p-tau-p2}
|\nabla g(x)|^p |Q'| \ge  (1 - \eps_1) \int_{Q'} |\nabla g |^p \, dy \ge  (1 - \eps) \int_{Q'} |\nabla g |^p \, dy. 
\end{equation}
Fix such a set $B$ and such an integer $\ell$. Set 
$$
B_m : = B \setminus A_m. 
$$
Since  $\mR^d \setminus (B \setminus A_m) \subset (\mR^d \setminus B) \cup A_m$, it follows  that 
\begin{equation*}
\int_{\mR^d \setminus B_m} | \nabla g|^p \, dx
=  \int_{\mR^d \setminus (B \setminus A_m)} | \nabla g|^p \, dx
\le \int_{\mR^d \setminus B } | \nabla g|^p \, dx
+ \int_{A_m} | \nabla g|^p \, dx. 
\end{equation*}
We deduce from \eqref{G1-p-pro-Am-11} and \eqref{estBm-11} that 
\begin{equation}\label{estBm}
\int_{\mR^d \setminus B_m} | \nabla g|^p \, dx
\le  \eps \int_{\mR^d}
|\nabla g|^p \, dx.  
\end{equation}

Set $P_\ell = \frac{1}{\ell} \mZ^d$. Let ${\bf \Omega}_\ell$ be the collection of all open
cubes with side length $1 / \ell$ whose vertices  belong to $P_\ell$ and denote 
\begin{equation*}
{ \bf
J}_\ell = \Big\{  Q' \in {\bf \Omega}_\ell ; \; Q'
\cap B_m \neq \emptyset \Big\}.
\end{equation*}
Take $Q' \in {\bf J}_\ell$ and $x \in Q' \cap B_m$. Since $g_k \to g$ in $L^p(Q')$, from 
\eqref{G1-p-rho-p1}, we obtain, for large $k$, 
\begin{equation*}
\frac{1}{|Q'|^{p/d}} \fint_{Q'} \big|g_k(y) - g(x) - \langle \nabla g(x), y - x \rangle  \big|^p \, dy < \hat \delta_1/ m^p \le \hat \delta_1 |\nabla g (x)|^p,  
\end{equation*}
since $|\nabla g(x)| \ge 1/m$ for $x \in B_m \subset \mR^d \setminus A_m$.   
Next, we apply Lemma~\ref{lem-Lp} with $\widetilde Q = Q'$, $h= g_k$, $a = \nabla g(x)$, $b= g(x)$, and large $k$;   we have 
\begin{equation*}
\Lambda_{\delta}(g_k, Q')  \ge
(\C- \eps) |\nabla g(x)|^p |Q'| \mbox{ for } \delta \in (0, \hat \delta_2 |\nabla g(x)|^p |Q'|^{1/d}),
\end{equation*}
which implies, by \eqref{G1-p-tau-p2},
\begin{equation}\label{estI'}
\liminf_{k \to + \infty}\Lambda_{\delta_k}(g_k, Q') \ge
(\C- \eps) (1- \eps)\int_{Q'} |\nabla g|^p \, dy.
\end{equation}
Since
\begin{align*}
\liminf_{k \to + \infty} \Lambda_{\delta_k}(g_k, \mR^d) \ge  \sum_{Q' \in
{\bf J}_\ell }\liminf_{k \to + \infty} \Lambda_{\delta}(g_k, Q'),
\end{align*}
it follows from \eqref{estI'} that
\begin{multline*}
\liminf_{k \to + \infty} \Lambda_{\delta_k}(g_k, \mR^d)
\ge (\C- \eps) (1- \eps) \sum_{Q' \in {\bf J}_\ell} \int_{Q'} |\nabla g|^p \, dx \\[6pt]
\ge  (\C- \eps) (1- \eps) \int_{B_m} |\nabla g|^p \, dx \mathop{\ge}^{\eqref{estBm}} 
 (\C- \eps) (1- \eps)^2 \int_{\mR^d} |\nabla g|^p \, dx; 
\end{multline*}
in the second inequality, we have  used the fact
$B_m$ is contained in $  \bigcup_{Q' \in {\bf J}_\ell} Q'$ up to a null set. 
Since $\eps>0$ is arbitrary, one has
\begin{align*}
\liminf_{k \to + \infty}\Lambda_{\delta_k}(g_k, \mR^d) \ge
\C \int_{\mR^d} |\nabla g|^p \, dx.
\end{align*}
The proof is complete. 
\end{proof}

\subsection{Proof of Property (G2)}

The proof of Property (G2) for $p>1$ is the same as the one for $p=1$ given in \cite{BrNg-Nonlocal1}. The details are omitted.  

\bigskip 
\noindent {\bf Acknowledgments.} This work was completed during a visit of H.-M. Nguyen at Rutgers University. He thanks H. Brezis for the invitation and  the Department of  Mathematics  for its hospitality.

\bigskip
\bigskip 
\noindent {

%
%

\centerline {\bf REFERENCES}}
\begin{thebibliography}{1}
\expandafter\ifx\csname url\endcsname\relax
  \def\url#1{\texttt{#1}}\fi
\expandafter\ifx\csname urlprefix\endcsname\relax\def\urlprefix{URL }\fi
\expandafter\ifx\csname href\endcsname\relax
  \def\href#1#2{#2} \def\path#1{#1}\fi

\bibitem{BrNg-Nonlocal1}
H.~Brezis, H.-M. Nguyen, {Non-local functionals related to the total variation
  and applications in Image Processing}, Ann. PDE 4 (2018) 77 pp.

\bibitem{BourNg}
J.~Bourgain, H.-M. Nguyen, {A new characterization of Sobolev spaces}, C. R.
  Acad. Sci. Paris 343 (2006) 75--80.

\bibitem{NgSob1}
H.-M. Nguyen, {Some new characterizations of Sobolev spaces}, J. Funct. Anal.
  237 (2006) 689--720.

\bibitem{NgGammaCRAS}
H.-M. Nguyen, {$\Gamma$-convergence and Sobolev norms}, C. R. Acad. Sci. Paris
  345 (2007) 679--684.

\bibitem{NgGamma}
H.-M. Nguyen, {$\Gamma$-convergence, Sobolev norms, and BV functions}, Duke
  Math. J. 157 (2011) 495--533.

\bibitem{SteinHarmonic}
E.~Stein, {Harmonic analysis: real-variable methods, orthogonality, and
  oscillatory integrals}, Vol.~43 of Princeton Mathematical Series, Princeton
  University Press, Princeton, 1993.

\bibitem{BrAnalyse1}
H.~Brezis, {Functional Analysis, Sobolev Spaces and Partial Differential
  Equations}, Springer, 2010.

\bibitem{Ziemer}
W.~Ziemer, {Weakly differentiable functions. Sobolev spaces and functions of
  bounded variation}, Vol. 120 of Graduate Texts in Mathematics,
  Springer-Verlag, New York, 1989.

\end{thebibliography}
\end{document}